\documentclass[a4paper]{article}

\usepackage[margin=1in]{geometry} 
\usepackage{amsmath, amsthm}
\usepackage{paralist}
\usepackage{graphics}
\usepackage{epsfig}
\usepackage{lscape}
\usepackage{geometry}
\usepackage{subfig}
\usepackage{amssymb}
\usepackage{tikz}
\usepackage{xcolor}
\usepackage{rotating}
\usepackage[ruled]{algorithm2e}
\usepackage{algorithm2e}
\usepackage{algorithmicx}
\usepackage{float}
\usepackage{pgf}
\usepackage{array,multirow,makecell}
\usepackage{tabularx, makecell, caption}
\usepackage{float}
\usepackage{multicol}
\usepackage{xcolor}
\usepackage{multirow}
\usepackage{makecell}
\usepackage{booktabs}
{\huge }  \usepackage[noend]{algpseudocode}
\usepackage{graphicx}  \usepackage{epstopdf}
\usepackage{array}
\usepackage{setspace}
\usepackage[colorlinks=true]{hyperref}

\usepackage[utf8]{inputenc}
\hypersetup{
	unicode,
	pdfauthor={Author One, Author Two, Author Three},
	pdftitle={A simple article template},
	pdfsubject={A simple article template},
	pdfkeywords={article, template, simple},
	pdfproducer={LaTeX},
	pdfcreator={pdflatex}
}

\usepackage[sort&compress,numbers,square]{natbib}
\theoremstyle{plain}
\newtheorem{theorem}{Theorem}

\newtheorem{proposition}[theorem]{Proposition}

\theoremstyle{definition}
\newtheorem{definition}[theorem]{Definition}

\usepackage{graphicx, color}
\graphicspath{{fig/}}

\usepackage{mathrsfs} 

\usepackage{lipsum}

\title{Bi-Objective Optimization over the Efficient Set of Multi-Objective Integer Quadratic Problem.}
\author{BENCHEIKH Ali $^1$ \and MOULAÏ Mustapha $^2$ \and BADAOUI Ilias $^1$}

\date{
	$^1$ USTHB, Dept. of Operations Research, Faculty of Mathematics, BP32 El-Alia, 16111, Algeria \\
	$^2${USTHB, LaROMaD Laboratory, Faculty of Mathematics, BP32 El-Alia, 16111, Algeria.
}}
\begin{document}
	\maketitle
	
	\begin{abstract}
		In this paper, we present an exact algorithm for  optimizing two linear fractional over the efficient set of a multi-objective integer quadratic problem. This type of problems arises when two decision-makers, such as firms, each have a preference function to optimize over the efficient set of a multi-objective problem. The algorithm employs a branch-and-cut approach, which involves: (1) exploring the solution space using a branch-and-bound strategy in the decision space, and (2) eliminating inefficient solutions using a cutting plane technique with efficient cuts constructed from the non-increasing directions of objective functions. Additionally, integral tests are incorporated to further ensure the efficiency of the obtained solutions.We present a comprehensive example, accompanied by a step-by-step resolution, to demonstrate the functioning of the algorithm.
		
		\noindent\textbf{Subject class:} Primary: 90C29, 90C10, 90C20, 90C32; Secondary:  90C57.
	\end{abstract}
	
	\section{Introduction}
	Multi-objective optimization serves as a powerful tool for resolving problems characterized by multiple conflicting objectives that need simultaneously optimized within specific circumstances. This class of problems frequently manifests in various real-world domains, including economics, engineering, supply chain management, and product planning. However, the resolution of multi-objective problems $(MOP) $ results in a trade-off between the objectives referred to as the "efficient set". Many studies have addressed $(MOP)$ cases using diverse methods, we can cite research work in the case when the objectives are modeled by linear functions:  \textcolor{blue}{(Sylva, J. \& Crema, A.)}\cite{sylva2004method}, \textcolor{blue}{(Chergui, M. et al)}\cite{chergui2008solving},  \textcolor{blue}{(Boland, N. et al)}\cite{boland2016shape}
	, \textcolor{blue}{(Lokman, B. \& K\"{o}ksalan, M.)}  \cite{lokman2013finding}. Further, for the nonlinear case, we cite: \textcolor{blue}{(Costa, J.)}\cite{costa2007computing}, \textcolor{blue}{(Miettinen, K.)}\cite{miettinen2001some}, \textcolor{blue}{(Kornbluth, J. \& Steuer, R.)}\cite{kornbluth1981multiple},     \textcolor{blue}{(Oua\"{i}l, F. \& Chergui, M.)}\cite{ouail2018branch}, \textcolor{blue}{(Chergui, M. \& Moula\"{i}, M.)}\cite{chergui2008exact}.

	Nevertheless, the search for generating all efficient solutions in $(MOP)$ faces operational issues in certain scenarios and encounters two specific challenges: (1) an excess of solutions, confuse decision-makers (DM) for selecting the best solution according to their preferences, and (2) computational burdens, especially when dealing with restrictions imposed on the $MOP$ model (e,g non-linear model, discrete decision variables). To address these issues, decision-maker preferences are modeled with a \emph{preference function} to be optimized over the effective set of $(MOP)$, allowing for targeted selection of the optimal solution without exhaustive enumeration. This concept is called "optimization over the effective set" and has been receiving more attention from researchers recently. In fact, we can cite for example studies where the $(MOP) $ and the preference function are linear functions: (continuous and discrete): \textcolor{blue}{(Benson, H.)}\cite{benson1984optimization}, \textcolor{blue}{(Ecker, J. \& Song, J.)} \cite{ecker1994optimizing}, \textcolor{blue}{(Jorge, J. )} \cite{jorge2009algorithm},  \textcolor{blue}{(Boland, N. et al)}\cite{boland2017new},  \textcolor{blue}{(Lokman, B. )} \cite{lokman2019optimizing}. However, the preference function (and/or) the multi-objective problem are not always linear which makes the solution process more complicated, we cite for example but not limited to: (\textcolor{blue}{Ozlen,M. et al}) \cite{ozlen2013optimising} who proposed an algorithm to optimize a nonlinear preference function in multi-objective integer programming. Similarly, \textcolor{blue}{( Moula\"{\i}, M. \& Drici, W.)} \cite{moulai2018indefinite} provided an algorithm for indefinite quadratic optimization on an efficient integer set, and more recently, \textcolor{blue}{(Chaiblaine, Y. \& Moula\"{\i}, M.)}\cite{chaiblaine2021exact} proposed an algorithm for the optimization of a quadratic function on the efficient set of multi-objective convex problems in fractional integers.

	From a decision-making perspective, consider a scenario involving two associated decision-makers, each expressing their own preference function and aiming to optimize it over a given efficient set. The solution to this problem yields a subset of the efficient set.
	This type of problem was first introduced in the literature by \textcolor{blue}{(Cherfaoui, Y. \& Moula\"{\i}, M.)}\cite{cherfaoui2021biobjective}, who proposed an algorithm for optimizing two linear preference functions within a linear multi-objective integer problem $(BL/MOILP)$. Subsequently,\textcolor{blue}{( Chaiblaine, Y. et al)}\cite{chaiblaine2020exact} addressed the same problem, focusing on the optimization of two fractional linear preference functions within a multi-objective fractional linear integer problem $(BLF/MOILFP)$. Following these in this paper, we extend the scope by considering the optimization of two fractional linear preference functions  over the efficient set of a multi-objective integer quadratic problem denoted $(BLF/MOIQP)$.

	The proposed algorithm is based on the branch-and-cut principle, navigating the search tree to obtain optimal solutions through the iterative resolution of more constrained linear fractional programs within an augmented Simplex tableau.  At each iteration, either a branching process is applied to reach the integrality of decision variables. Or, if the current optimal solution already possesses integer values,  it is subjected to efficiency tests to evaluate its potential for inclusion among the $(BLF/MOIQP)$ solutions. To speed up the process,  we exploit the decreasing directions of criteria from both the $(MOIQP)$ \cite{ouail2018branch} and the $(BOILFP)$ \cite{chergui2008exact} to construct efficient cuts. These cuts effectively eliminate a significant number of inefficient integer solutions, avoiding the need to enumerate all efficient solutions for both problems.
	
	The rest of the article is organized as follows: in section 2, the problem formulation and notations are given, followed by the theoretical background and main results in section 3.
	In section 4, we describe the proposed algorithm and provide an illustrated example. Finally, we conclude the article and discuss perspectives in the $5^{th}$ section.
	
	\section{Problem formulation \& Notations}
	\subsection{Problem formulation}
	
	Consider $ r \;  ( r \ge 2 )$ convex quadratic objective functions of the form:
	$$ f_i(x)=\frac{1}{2}x^TQ_ix+c_i^Tx \quad i=1,\ldots,r$$
	and define the set
	$\mathcal{D} =\mathcal{X}\cap \mathbb{Z}^n, \quad \text{where} \quad \mathcal{X}=\{x\in \mathbb{R}_+^n | Ax\le b  \}.$\\
	The multi-objective integer quadratic problem $(MOIQP)$ is then formulated as
	\[ (MOIQP)\begin{cases}
		\text{"Min"}\, f_i(x) \quad i=1,\ldots,r\\x\in \mathcal{D}
	\end{cases}
	\]
	where:
	\begin{itemize}
		\item $Q_i \in \mathbb{Z}^{n\times n},\; i=1,\ldots,r$ are positive semi-definite matrices, and $c_i \in \mathbb{Z}^n,\;i=1,\ldots,r$ are vectors.
		\item $A\in  \mathbb{Z}^{m\times n}$ is a matrix, and $b \in \mathbb{Z}^m$ is a vector.
		\item $\mathcal{X}$ is assumed to be a closed, bounded, non-empty convex polyhedron.
		
	\end{itemize}
	We denote the efficient set of $(MOIQP)$ by $\mathcal{X}_{\mathcal{Q}}$.
	The bi-objective integer fractional program $(BOILFP)$, on the other hand, is modeled as:
	\[ (BOILFP)\begin{cases}
		\text{"Min"}\, \psi^s(x)=\dfrac{p^sx+\alpha^s}{q^sx^T+\beta^s} \quad s=1,2
		\\x\in \mathcal{D}
	\end{cases} \]
	where:
	\begin{itemize}
		\item $\psi^s(x)$ is a linear fractional preference function, $ s=1,2$.
		\item $p^s,\,q^s$ are real $n$-vectors
		\item $\alpha^s,\,\beta^s$ are real constants.
		
	\end{itemize}
	The efficient set of $(BOILFP)$ is denoted by $\mathcal{X}_{\mathcal{F}}$.\\
	Correspondingly, the main problem addressed in this work is the bi-objective linear fractional optimization over the efficient set of $(MOIQP)$, stated as:
	\[(BLF/MOIQP) \begin{cases} \text{"Min"}\, \psi^s(x) \quad s=1,2\\
		x \in \mathcal{X}_{\mathcal{Q}} \end{cases}.\]
	The goal is to find the set  $\mathcal{X}_{Eff}\subset \mathcal{X}_{\mathcal{Q}}$ of solutions that are efficient with respect to both $(MOIQP)$ and $(BOILFP)$.
	
	We assume that: \begin{itemize}
		\item $\psi^s(x)$ is not a strictly positive combination of linear parts of quadratic function objectives.
		\item The factors $q^sx^T+\beta^s, \; s=1,2$ are positive for all $x \in\mathcal{X}$.
		\item   $\mathcal{X}_{\mathcal{Q}} \neq \mathcal{D}$, ensuring the problem is non-trivial.
	\end{itemize}
	\begin{definition}
		A feasible solution $x^*$ is said to be \emph{efficient} in $(MOIQP)$ if, and only if, there does not exist $x \in \mathcal{D}$ such that $f_i(x)\le f_i(x^*)$ for all $i=1,\ldots,r$, and $f_i(x)<f_i(x^*)$ for at least one index $i$.
	\end{definition}
	
	The adopted strategy involves generating an optimal integer solution for the main problem $(BLF/MOIQP)$ by  solving a sequence of continuous, more constrained linear fractional problems $(LFP)_l$ at each iteration $l,\,l\ge 0$, defined as:
	\[(LFP)_l \begin{cases} \text{"Min"}\, \psi^1(x)=\dfrac{p^1x^t+\alpha^1}{q^1x^t+\beta^1} \\
		x \in \mathcal{X}_{l} \end{cases}\]
	where $\mathcal{X}_0=\mathcal{X}$ and $\mathcal{X}_{l+1}\subset \mathcal{X}_l$ is to be explored at iteration $l$.
	
	\subsection{Notations}
	To facilitate a better understanding of the theoretical results, we provide essential notations:
	
	We denote by $ \mathcal{B}_l$ and $ \mathcal{N}_l$ the index sets of basic and non-basic variables in the simplex table, respectively.
	
	$x^{*(l)}$ represents the optimal solution of the problem $(LFP)_l$ corresponding to the basis $B_l$ obtained at node $l$ in the tree. Additionally, we denote:
	
	\begin{itemize}
		\item \textbf{For objective functions of $(MOIQP)$:}
		The functions $ f_i, i=1,\ldots,r $ can be expressed in the neighborhood of $x^{*(l)} $ as:
		\[f_i(x)=f_i(x^{*(l)})+\nabla f_i(x^{*(l)})(x-x^{*(l)})+\epsilon_i \lVert x-x^{*(l)}\rVert
		\]
		where $  \nabla f_i(x^{*(l)})= Q_ix^{*(l)}+c_i \; \text{and }\epsilon_i:\mathbb{R}^n\longrightarrow \mathbb{R}$.
		
		The linear expression $ \nabla f_i(x^{*(l)})(x-x^{*(l)}) $ helps determine the non-increasing criterion directions of the edges of the relaxed feasible region originating from the extreme point $x^{*(l)}$.
		
		\item \textbf{For preference functions of $(BOILFP)$}:
		We also denote by $\bar{\gamma}^s_j \;\text{for}\,s=1,2$ the $j^{th}$ component of the reduced gradient of the vector $\bar{\gamma^s}$ of the preference function $\psi^s$ at each iteration $l$ with the following relation:
		\[\bar{\gamma}^{s}_j=\mathcal{Q}^{s}(x)\eta_{j}^s-\mathcal{P}^{s}(x)\vartheta_{j}^s .\]
		where:
		\begin{itemize}
			\item  $\eta_{j}^s =p^{s}_j-p^s_{\mathcal{B}_l}\mathcal{B}_l^{-1}a_j;\;\vartheta_{j}^s =q^{s}_j-q^s_{\mathcal{B}_l}\mathcal{B}_l^{-1}a_j;$ and $a_j$ is the $j^{th}$ column of matrix $A$.
			\item $\mathcal{P}^s(x^{*(l)})=p^{sT}x^{*(l)}+\alpha^s$ and $\mathcal{Q}^s(x^{*(l)})=q^{sT}x^{*(l)}+\beta^s$.
			\item  $\psi^s=\mathcal{P}^s/\mathcal{Q}^s$.
		\end{itemize}

		\item \textbf{For the optimal simplex table}:
		For the basic variables $x_k$ of $x$ we have the equation $x_k = \hat{b}_{p(k)}- \sum_{j=1}^{n} \hat{a}_{p(k)j} x_j $ for all $ k \in \mathcal{B}_l$, $p(k)$ where: ${p(k)}$ indicates the position of $ x_k$ in $x$ , $\hat{a}_{p(k)j}$ and $\hat{b}_{p(k)}$ indicate the updated values, in the simplex
		table, of elements of the constraint matrix $A$ and vector $b$, respectively.\\
		\item \textbf{Sets and Used cuts}:
		We have: $ \mathcal{N}_l \cap \overline{\{1,n\}}$ as the index set of non-basic original variables, $ \mathcal{N}_l \backslash \overline{\{1,n\}} $ as the index set of non-basic surplus variables, and $ \mathcal{B}_l \cap \overline{\{1,n\}} $ as the index set of basic original variables. The set $ \mathcal{H}_l$ indicates non-increasing directions of the criteria $f_i, i=1,\ldots,r$.
		:
		$$  \mathcal{H}_l=\left\{ j\in \mathcal{N}_l| \exists i\in \{ 1,\ldots,r\}; \left( \rho_j-\sum_{k\in \mathcal{B}_l\cap\overline{\{1,n\}}}\nabla f_i(x^{*(l)})_k\hat{a}_{p(k)j} \right)<0 \right\} $$ \\$$ \bigcup \left\{ j\in \mathcal{N}_l| ;\left( \rho_j-\sum_{k\in \mathcal{B}_l\cap\overline{\{1,n\}}}\nabla f_i(x^{*(l)})_k\hat{a}_{p(k)j} \right)=0 ,\;\forall i\in \{ 1,\ldots,r\}\right\} $$
		where $ \nabla f_i(x^{*(l)})_k$ is the $ k^{th}$ coordinate of the gradient  vector of criterion $ i$ in  $ x^{*(l)}$ and $ \rho_j$ is defined by :
		\[\rho_j=\begin{cases} \nabla f_i(x^{*(l)})_k \quad \text{if} \; j\in \mathcal{N}_l\cap\overline{\{1,n\}} \\
			0  \quad \qquad \qquad \; \text{if}\; j \in \mathcal{N}_l\backslash \overline{\{1,n\}}\end{cases}\]
		However, in the optimal simplex tables, we denote:$$ \bar{f}_i=\rho_j-\sum_{k\in \mathcal{B}_l\cap\overline{\{1,n\}}}\nabla f_i(x^{*(l)})_k\hat{a}_{p(k)j} \; \text{for} \;i=1,\ldots,r$$.\\
		Our method is based on  \emph{Branch and Cut} principle, we use the information to construct the efficient cut in the aim to remove all non efficient integer solutions for $(MOIQP)$. For this, we define the efficient cut as follow:
		\begin{equation}
			\label{eq1}	
			\sum_{j \in \mathcal{H}_l} x_j \ge 1
		\end{equation}
		Also, we define the set $ \mathcal{H'} $
		\[ \mathcal{H'}=\left\{j \in \mathcal{N}|\bar{\gamma}^2_j<0 \right\}\bigcup\left\{ j \in \mathcal{N}|\bar{\gamma}^1_j=0 \,\text{and}\, \bar{\gamma}^2_j=0 \right\} \]
		then , we can define the following cut:
		\begin{equation}
			\label{eq2}	
			\sum_{j \in \mathcal{H'}_l} x_j \ge 1
		\end{equation}
		At the node $l$, we can define the following subsets:
		\[\mathcal{X}_{l+1}^1=\{x\in \mathcal{X}_l\mid \sum_{j \in \mathcal{H}_l} x_j \ge 1\} \]
		\[ \mathcal{X}_{l+1}^2=\{ x \in \mathcal{X}_l \mid \sum_{j \in \mathcal{H'}_l} x_j \ge 1\ \} \]
		and
		\[ \mathcal{X}_{l+1}=\mathcal{X}^1_{l+1}\cup \mathcal{X}^2_{l+1}\]
	\end{itemize}
	\section{Theoretical background \& Main results}
	In this section, we present important definitions and theoretical results that are essential for understanding the subsequent sections\\
	Let be given the following theorem:
	\begin{theorem}
		The feasible solution $ x^{*(l)}$ is an optimal solution for the problem $ (LFP)_l$ if and only if the vector $ \gamma$ is such that $ \bar{\gamma_j} \ge 0$ for all $ j \in \mathcal{N}_l.$ see \cite{martos1975nonlinear}
	\end{theorem}
	
	\begin{theorem}
		Supposing that $ \mathcal{H}_l\ne \emptyset$ and $ \mathcal{H'}_l\ne \emptyset$ at the current integer solution $ x^{*(l)}$. If $ x\ne x^{*(l)}$ is an integer efficient solution in domain   $\mathcal{X}_l \backslash \{ x^{*(l)}\}$, then $ x \in \mathcal{X}_{l+1}$.
		
	\end{theorem}
	\begin{proof}
		Consider any $x \in \mathcal{X}_l$ such that, $x\ne x^{*(l)}$. Suppose that $ x\notin\mathcal{X}_{l+1}$. Then $x \notin \mathcal{X}^{1}_{l+1} $ and $x \notin \mathcal{X}^{2}_{l+1} $.
		\begin{itemize}
			
			\item If $x \notin \mathcal{X}^{1}_{l+1}$, then $ \sum_{j \in \mathcal{H}_l}x_j < 1 $. This implies that $x_j=0$ for all $ j \in \mathcal{H}_l $, as $x$ is an integer. However, when considering only the first $n$ variables of $(MOIQP)$ in the solution $x^{*(l)}$ and $x$, we can calculate the following:
			\[\nabla f_i(x^{*(l)})(x-x^{*(l)})=-\nabla f_i(x^{*(l)})x^{*(l)}+\nabla f_i(x^{*(l)})x \quad \forall i \in \{1,\ldots,r\}.\]
			
			Furthermore, from the current optimal simplex table at $x^{*(l)}$, we have the equation:
			$x_k=\hat{b}_{p(k)}-\sum_{j \in \mathcal{N}_l} \hat{a}_{p(k)j}x_j$ for all index $ k\in \mathcal{B}_l$. Now, we can write the following:
			\begin{equation*}
				\begin{split}
					\nabla f_i(x^{*(l)})x &= \sum_{k\in \mathcal{B}_l\cap\overline{\{1,n\}}}\nabla f_i(x^{*(l)})_k(\hat{b}_{p(k)}-\sum_{j \in \mathcal{N}_l} \hat{a}_{p(k)j}x_j)\\
					& + \sum_{j\in \mathcal{N}_l\cap\overline{\{1,n\}}}\nabla f_i(x^{*(l)})_jx_j \qquad \quad \forall i \in \{1,\ldots,r\}
				\end{split}
			\end{equation*}
			
			\begin{equation*}
				\begin{split}
					\nabla f_i(x^{*(l)})x&= \sum_{k\in \mathcal{B}_l\cap\overline{\{1,n\}}}\nabla f_i(x^{*(l)})_k\hat{b}_{p(k)} \\
					&- \sum_{j \in \mathcal{N}_l}\left[  \sum_{k\in \mathcal{B}_l\cap\overline{\{1,n\}}} (\nabla f_i(x^{*(l)}))_k\hat{a}_{p(k)j} \right]x_j\\
					& + \sum_{j\in \mathcal{N}_l\cap\overline{\{1,n\}}} \nabla f_i(x^{*(l)})_jx_j \qquad\quad \forall i \in \{1,\ldots,r\}
				\end{split}
			\end{equation*}
			
			Besides, we have:
			\[\nabla f_i(x^{*(l)})x^{*(l)}=\sum_{k\in \mathcal{B}_l\cap\overline{\{1,n\}}} \nabla f_i(x^{*(l)})_k\hat{b}_{p(k)} \]
			Thus, we obtain:
			\begin{equation*}
				\begin{split}
					\nabla f_i(x^{*(l)})(x-x^{*(l)})&=\sum_{j\in \mathcal{N}_l\cap\overline{\{1,n\}}} \left[ \nabla f_i(x^{*(l)})_j-\sum_{k\in \mathcal{B}_l\cap\overline{\{1,n\}}}\nabla f_i(x^{*(l)})_k\hat{a}_{p(k)j} \right]x_j \\
					&+ \sum_{j\in \mathcal{N}_l\backslash\overline{\{1,n\}}} \left[ -\sum_{k\in \mathcal{B}_l\cap\overline{\{1,n\}}}\nabla f_i(x^{*(l)})_k\hat{a}_{p(k)j}
					\right]x_j\quad \forall i \in \{1,\ldots,r\}
				\end{split}
			\end{equation*}
			
			Therefore:
			\[ \nabla f_i(x^{*(l)})(x-x^{*(l)})=\sum_{j \in \mathcal{N}_l}\left[ \rho_j-\sum_{k\in \mathcal{B}_l\cap\overline{\{1,n\}}}f_i(x^{*(l)})_k\hat{a}_{p(k)j} \right]x_j \quad \forall i \in \{1,\ldots,r\}.  \]
			
			Where:
			\[ \rho_j=\begin{cases} \nabla f_i(x^{*(l)})_k \quad \text{if} \; j\in \mathcal{N}_l\cap\overline{\{1,n\}} \\
				0  \quad \qquad \qquad \; \text{if}\; j \in \mathcal{N}_l\backslash \overline{\{1,n\}}\end{cases}\]
			
			Using the set $\mathcal{H}_l$ as a subset of the $\mathcal{N}_l$ set, we can write:
			\begin{multline*}
				\nabla f_i(x^{*(l)})(x-x^{*(l)})=\sum_{j \in \mathcal{H}_l} \left[ \rho_j-\sum_{k\in \mathcal{B}_l\cap\overline{\{1,n\}}}\nabla f_i(x^{*(l)})_k\hat{a}_{p(k)j}\right]x_j \\
				+ \sum_{j \in \mathcal{N}_l\backslash \mathcal{H}_l}\left[ \rho_j-\sum_{k\in \mathcal{B}_l\cap\overline{\{1,n\}}}\nabla f_i(x^{*(l)})_k\hat{a}_{p(k)j}\right]x_j \quad \forall i \in \{1,\ldots,r\}
			\end{multline*}
			
			As it has been assumed that $x_j=0, \, \forall j \in \mathcal{H}_l $, then the last expression is reduced to:
			\[ \nabla f_i(x^{*(l)})(x-x^{*(l)})=\sum_{j \in \mathcal{N}_l\backslash \mathcal{H}_l}\left[ \rho_j-\sum_{k\in \mathcal{B}_l\cap\overline{\{1,n\}}}\nabla f_i(x^{*(l)})_k\hat{a}_{p(k)j}\right]x_j \quad \forall i \in \{1,\ldots,r\} \]
			
			According to the definition of set $\mathcal{H}_l$, for all index $j \in \mathcal{N}_l\backslash \mathcal{H}_l$, we have the following inequality:
			\[\rho_j-\sum_{k\in \mathcal{B}_l\cap\overline{\{1,n\}}}\nabla f_i(x^{*(l)})_k\hat{a}_{p(k)j}\ge 0\quad \forall i \in \{1,\ldots,r\} \]
			with at least one strict inequality. Hence:
			\[\nabla f_i(x^{*(l)})(x-x^{*(l)})\ge 0\quad \forall i \in \{1,\ldots,r\}\]
			
			Then, we obtain the following inequality:
			\[f_i(x^{*(l)})+\nabla f_i(x^{*(l)})(x-x^{*(l)})\ge f_i(x^{*(l)})  \quad\forall i \in \{1,\ldots,r\}\]
			
			From the previous inequality, we conclude that:
			\[ f_i(x)\ge f_i(x^{*(l)}) \quad\forall i \in \{1,\ldots,r\} \]
			
			Since there is at least one strict inequality, we deduce that $f(x)$ is dominated by $f(x^{*(l)})$, implying that $x$ is not an efficient solution.
			
			\item If $x \notin \mathcal{X}^2_{l+1}$, then $x \in \{x \in \mathcal{X}_l | \sum_{j \in \mathcal{N}_l \backslash \mathcal{H'}_l}x_j \ge 1\}$; similarly, it implies that
			\[x \in \{x \in \mathcal{X}_l | \sum_{j \in \mathcal{H'}_l} < 1\}\]
			The following inequalities hold:
			\[\sum_{j \in \mathcal{H'}_l}x_j < 1\]
			\[\sum_{j \in \mathcal{N}_l \backslash \mathcal{H'}_l}x_j \ge 1\]
			
			It follows that $x_j = 0$ for all $j \in \mathcal{H'}_l$ and $x_j \ge 1$ for at least one index $j \in \mathcal{N}_l \backslash \mathcal{H}_l$.
			
			The updated values for each objective function from the optimal tableau simplex corresponding to the solution $x^{*(l)}$ are written according to the non-basic indexes $j \in \mathcal{N}_l$ as follows:
			\[\mathcal{P}^{s}(x) = \mathcal{P}^{s}(x^{*(l)}) + \theta_j\eta^{s(l)}_j\]
			\[\mathcal{Q}^{s}(x) = \mathcal{Q}^{s}(x^{*(l)}) + \theta_j\vartheta^{s(l)}_j\]
			where $\theta_j = \dfrac{x^{*(l)}_{\mathcal{B}_{l_{(r*)}}}}{A^r_j} = \min \{\dfrac{x^{*(l)}_{\mathcal{B}_{l_{(r*)}}}}{A^r_j} | A^r_j > 0\}$.
			
			Therefore, we have:
			\[\psi^2(x) = \dfrac{\mathcal{P}^2(x)}{\mathcal{Q}^2(x)} = \dfrac{\mathcal{P}^2(x^{*(l)}) + \delta_j\eta_j}{\mathcal{Q}^2(x^{*(l)}) + \delta_j\vartheta_j}\]
			
			Then, we have
			\[\psi^2(x) - \psi^2(x^{*(l)}) = \dfrac{\mathcal{P}^2(x^{*(l)}) + \delta_j\eta_j}{\mathcal{Q}^2(x^{*(l)}) + \delta_j\vartheta_j} - \dfrac{\mathcal{P}^2(x^{*(l)})}{\mathcal{Q}^2(x^{*(l)})}\]
			\[= \delta_j\dfrac{[\mathcal{Q}^2(x^{*(l)})(\eta^{2(l)}_j) - \mathcal{P}^2(x^{*(l)})(\vartheta^{2(l)}_j)]}{\mathcal{Q}^2(x^{*(l)})[\mathcal{Q}^2(x^{*(l)}) + \delta_j(\vartheta^{2(l)}_j)]}\]
			\[= \delta_j\dfrac{[\mathcal{Q}^2(x^{*(l)})(\eta^{2(l)}_j) - \mathcal{P}^2(x^{*(l)})(\vartheta^{2(l)}_j)]}{\mathcal{Q}^2(x^{*(l)})\mathcal{Q}^2(x)}\]
			
			Since we have the following notation:
			\[\bar{\gamma}^{2(l)}_j = \mathcal{Q}^{2}(x^{*(l)})\eta_{j}^{2(l)} - \mathcal{P}^{2}(x^{*(l)})\vartheta_{j}^{2(l)}.\]
			As the component ${\gamma}^{2(l)}_j \ge 0$ for all $j \in \mathcal{N}\backslash\mathcal{H'}$,
			\[\psi^2(x) - \psi^2(x^{*(l)}) \ge 0\]
			
			Thus, $\psi^2(x^{*(l)}) \le \psi^2(x)$, with $\psi^1(x^{*(l)}) \le \psi^1(x)$; then $\psi(x)$ is dominated by $\psi(x^{*(l)})$, however, $x \notin \mathcal{X}_{\mathcal{Q}}$; hence $x \notin \mathcal{X}_{Eff}$.
			
		\end{itemize}
	\end{proof}
	\begin{proposition}
		\label{prop2}
		If $ \mathcal{H}_l=\emptyset$ or $\mathcal{H'}_l=\emptyset $ at the current integer solution $ x^{*(l)}$, then $ \mathcal{X}_l \backslash \{ x^{*(l)}\}$ is an explored domain
	\end{proposition}
	\begin{proof}
		\begin{itemize}
			\item Assume $\mathcal{H}=\emptyset$, then $ \forall j \in  \mathcal{N}$, we have $\bar{f}^i_j \ge 0 $ and $ \exists i_0 \in {1,\ldots, r}$ such that $\bar{f}^{i_0}_j > 0,\, \forall j \in \mathcal{N}_l$. So $x^{*(l)} $ dominates  all points $x, \,\text{and }\, x\neq x^{*(l)} $ of domain $\mathcal{D}_l $
			\item Now assume that $\mathcal{H'}_l $ then $ \forall \in \mathcal{N}_l$ , $\bar{\gamma}^2_j>0 $ or  $ \bar{\gamma}^2_j=0 $ and $\bar{\gamma}^1_j=0 $ in addition to that $ \bar{\gamma}^1_j>0,\, \forall j\in \mathcal{N}_l $ since it is an optimal solution for $ (LFP)_l$, $x^{*(l)} $ becomes the most preferred solution in the domain $\mathcal{D}_l $
		\end{itemize}
	\end{proof}
	\begin{theorem}
		The algorithm terminates in  a finite number of iterations and the set $\mathcal{X}_{Eff} $ contains all the solutions of  $(BLF/MOIQP) $ if such solutions exist.
	\end{theorem}
	\begin{proof}
		As $\mathcal{D}$, the set of integer feasible solutions of $(MOIQP)$, is a finite and bounded set contained within $\mathcal{X}$, the cardinalities of the efficient sets $\mathcal{X}_{Eff}$, $\mathcal{X}_{\mathcal{Q}}$, and $\mathcal{X}_{\mathcal{F}}$ are also finite. These sets comprise a limited number of integer solutions, implying that the search tree will have a finite number of branches. Therefore, the algorithm terminates after a finite number of steps.
		
		For $\mathcal{X}_{Eff}$ to encompass all solutions of $(BLF/MOIQP)$, fathoming rules are employed to ensure no elements in $\mathcal{X}_{Eff}$ are lost. At each step $l$ of the algorithm \ref{Algo}, if an integer solution $x^{*(l)}$ is found, the cuts eliminate $x^{*(l)}$ and all dominated solutions from the search (see proposition [\ref{prop2}]).
		
		The first fathoming rule applies when the set $\mathcal{H}_l$ or $\mathcal{H'}_l$ is empty. In this case, the current node can be pruned, as the remaining domain contains only solutions dominated either in terms of the $(MOIQP)$ program or the $(BOILFP)$. The second rule is the trivial case where the reduced domain becomes infeasible, whether due to previous cuts or branching.
		
		Moreover, let \(x_l\) represent an integer optimal solution explored during iteration \(l\). Various scenarios can arise for \(x_l\) during the algorithm's execution:
		\begin{itemize}
			\item If \(x_l\) does not satisfy the efficiency criteria for the $(MOIQP)$ problem, it is initially excluded.
			\item Conversely, if \(x_l\) remains efficient with respect to the $(MOIQP)$ problem, it becomes a candidate for inclusion in the set $\mathcal{X}_{Eff}$ for the $(BLF/MOIQP)$ problem. Assuming all candidate solutions for the $(BLF/MOIQP)$ problem, there will be at least one solution that is efficient for the $(BOILFP)$ problem. This scenario leads to the exclusion of all solutions that are dominated by such a solution and are inefficient in terms of $(BOILFP)$. Thus, efficiency tests are conducted for both $(MOIQP)$ and $(BOILFP)$ to establish their respective criteria.
		\end{itemize}
	\end{proof}
	
	\subsection{Efficiency  Tests}
	Let $x^{*(l)}$ be an optimal solution of the program $(LFP)_l$. There are two efficiency tests to take into account. First, we test the efficiency of the obtained optimal solution for the program $(MOIQP)$ by solving the following single-objective mathematical program:
	\[ (\mathcal{T}^1_{x^{*(l)}})\begin{cases}
		\max \,\, \varphi = \sum\limits_{i=1}^{r}\varepsilon_i \\
		f_i(x) + \varepsilon_i \leq f_i(x^{*(l)}) \qquad i=1,\ldots, r \\
		\varepsilon_i \geq 0 \\
		x \in \mathcal{D}
	\end{cases}
	\]
	The point $x^{*(l)}$ is $\emph{efficient}$ (i.e., $x^{*(l)} \in \mathcal{X}_{\mathcal{Q}}$) for the original problem if the objective function $\varphi$ is null in $(\mathcal{T}^1_{x^{*(l)}})$. Otherwise, let $\hat{x}^{(l)}$ be the optimal solution of $(\mathcal{T}^1_{x^{*(l)}})$; then $\hat{x}^{(l)} \in \mathcal{X}_{\mathcal{Q}}$ (see \cite{benson1978existence}).
	
	However, the second efficiency test for the program $(BIOLFP)$ is by solving the following program (for proof, see \cite{chaiblaine2020exact}):
	\[ (\mathcal{T}^2_{x^{*(l)}})\begin{cases}
		\max \,\, w_1 + w_2 \\
		(p^1 - \psi^1(x^{*(l)})q^1)x + w_1 \leq \psi^1(x^{*(l)})\beta^1 - \alpha^1 \\
		(p^2 - \psi^2(x^{*(l)})q^2)x + w_2 \leq \psi^2(x^{*(l)})\beta^2 - \alpha^2 \\
		w_1, w_2 \geq 0 \\
		x \in \mathcal{D}
	\end{cases}
	\]
	If both objective functions for the programs $(\mathcal{T}^1_{x^{*(l)}})$ and $(\mathcal{T}^2_{x^{*(l)}})$ are null, it means that the solution $x^{*(l)}$ is efficient for $(MOIQP)$ and $(BOILFP)$, and thus $x^{*(l)} \in \mathcal{X}_{Eff}$.
	
	\section{Description of Method}
	At each iteration $l$, the program $(LFP)_l$ is solved. The corresponding node $l$ is fathomed if $(LFP)_l$ becomes infeasible due to either $\mathcal{H}_l$ or $\mathcal{H'}_l$ being empty. However, if the optimal solution $x^{*(l)}$ is not integer, a fractional coordinate $x_j = x_j^{*(l)}$ is identified. The feasible set $\mathcal{X}_l$ is then partitioned into subsets $\mathcal{X}_{l_1}$ and $\mathcal{X}_{l_2}$, creating two new nodes in the search tree with additional constraints: $x_j \le \lfloor x_j^{*(l)} \rfloor$ for $(LFP)_{l_1}$ and $x_j \ge \lfloor x_j^{*(l)} \rfloor + 1$ for $(LFP)_{l_2}$, where $l_1>1$ and $ l_2 > l$.
	
	If $x^{*(l)}$ is integer, its efficiency for $(MOIQP)$ is tested by solving $(\mathcal{T}^1_{x^{*(l)}})$. If the efficiency for $(MOIQP)$ is guaranteed, then  its efficiency for $(BOILFP)$ is tested by solving $(\mathcal{T}^2_{x^{*(l)}})$. If $x^{*(l)}$ efficient for both, the set $\mathcal{X}_{Eff}$ is updated to include $x^{*(l)}$. If $x^{*(l)}$ is not efficient for one or both tests, sets $\mathcal{H}_l$ and $\mathcal{H'}_l$ are constructed. Efficient cuts \ref{eq1} and \ref{eq2} are added to the successor nodes of $l$, resulting in $\mathcal{X}_{l+1} = \{x \in \mathcal{X}_l | \sum_{j \in \mathcal{H}_l} \ge 1, \sum_{j \in \mathcal{H'}_l} \ge 1\}$.
	
	\subsection{ BLF/MOIQP Algorithm}
	The algorithm steps are detailed in the code below:
	\begin{algorithm}[h]
		\SetKwFunction{blfmoiqp}{blfmoiqp}
		\SetKwInOut{KwIn}{Input}
		\SetKwInOut{KwOut}{Ensure}
		\KwOut{$ \mathcal{X}_{Eff}$ the solution of $(BLF/MOIQP)$ }
		$\bullet $ \textbf{Step 1:Initialization}:$ \mathcal{X}_{Eff}=\emptyset, l=0,\, \mathcal{X}_{0}=\mathcal{X}$.\\
		\While{there is a non fathomed node $l$ }{
			$\bullet $ solve $(LFP)_l $ using simplex or dual simplex method.\\
			\eIf{($(LFP)_l$) has an optimal solution $x^{*(l)}$ }{
				\eIf{$x^{*(l)}$ is integer}{Solve ($\mathcal{T}^1_{x^{*(l)}}$);
					\If{the optimal value of the objective function of ($\mathcal{T}^1_{x^{*(l)}}$) is $0$}{Solve ($\mathcal{T}^2_{x^{*(l)}}$);
						\If{the optimal value of the objective function of ($\mathcal{T}^2_{x^{*(l)}}$) is $0$}{$ \mathcal{X}_{Eff}=\mathcal{X}_{Eff}\cup \{x^{*(l)}\}$}}
					Construct the sets $\mathcal{H}_l $ and $ \mathcal{H'}_l$\\
					\eIf{$\mathcal{H}_l=\emptyset $ or $ \mathcal{H'}_l=\emptyset$}{Fathom the node $l$}{add the cuts \ref{eq1} and \ref{eq2} to the successors of $l$} .
					
				}{Choose the index $k$ such $ x^{*(l)}_k$ is fractional. Then ,split the program $(LFP)_l $ into two subproblems, by adding respectively the constraints $x_k=\lfloor x^{*(l)}_k \rfloor $ and $x_k=\lfloor x^{*(l)}_k \rfloor+1 $  to obtain $(LFP)_{l1} $ and $(LFP)_{l2}\, (l_1\ge l+1,l_2\ge l+1 )$ and $ l_1\neq l_2$}	
			}{fathom the node $l$.}}
		\caption{Bi-objective  optimization over multi-objective integer quadratic efficient set.}
		\label{Algo}
	\end{algorithm}
	
	\subsection{Illustrative example}
	To illustrate how the algorithm $(BLF/MOIQP)$ works, we provide an example to optimize a linear fractional function over the efficient set of a Tri-OIQP problem. The steps for solving the problem are detailed and followed by the solution search tree [\ref{Tree1}].
	
	\[ (BLF/MOIQP)\begin{cases}
		\text{Min}\, \psi^1(x)=\dfrac{x_1-4x_2-x_3-7}{x_1+x_3+3} \\
		\text{Min}\, \psi^2(x)=\dfrac{-2x_1+x_2-3x_3-2}{x_1+x_2+x_3+2} \\
		x\in \mathcal{X}_{\mathcal{Q}}
	\end{cases}
	\]
	Where:
	$ \mathcal{X}_{\mathcal{Q}} $ is the efficient set of the following (Tri-OIQP):
	\[ (MOIQP)\begin{cases}
		\text{Min}\, f_i(x)=\frac{1}{2}x^tQ_ix+c_i^tx \quad i=\overline{1,3}\\
		Ax\le b\\
		x\in \mathbb{R}_+^3\cap\mathbb{Z}^3
	\end{cases}
	\]
	\begin{equation*}
		Q_1 =
		\begin{pmatrix}
			50 & 43&20\\
			43 & 42 &20\\
			20&20 &11 \\
		\end{pmatrix}
		\;\; ,
		c^t_1 =
		\begin{pmatrix}
			-94 & -74&-37 \\
		\end{pmatrix}
		\qquad
		Q_2 =
		\begin{pmatrix}
			33 & 22&25\\
			22 & 21 &18\\
			25   &18 &42\\
		\end{pmatrix}
		, \;\;
		c^t_2 =
		\begin{pmatrix}
			6 &90&-37 \\
		\end{pmatrix}
	\end{equation*}
	\begin{equation*}
		Q_3 =
		\begin{pmatrix}
			50 & 17&43\\
			17 & 6&15\\
			43 &  15  &38
		\end{pmatrix}
		, \;\;
		c^t_3 =
		\begin{pmatrix}
			36 & -20&-70
		\end{pmatrix}
	\end{equation*}
	And
	\begin{equation*}
		A =
		\begin{pmatrix}
			1&1&1 \\
			-1&2&3 \\
		\end{pmatrix}
		\;\;
		x^t =
		\begin{pmatrix}
			x_1 & x_2 & x_3 \\
		\end{pmatrix}
		\;\;
		b =
		\begin{pmatrix}
			3 \\
			6
		\end{pmatrix}
	\end{equation*}
	\begin{itemize}
		\item \textbf{Initialization:} We put $l=0,\,\mathcal{X}_{Eff}=\emptyset$
		\item \textbf{node $N_0$: }we solve the program $(LFP)_0 $, the first optimal solution obtained $ x^{*(0)}=(0,3,0)$. The results are summarized in the following table:
		
		\begin{table}[H]
			\centering
			\begin{tabular}{ccccc}
				\toprule
				$\mathcal{B}_5$ & $x_1$ &$x_3$ &$x_5$&$RHS$ \\
				\midrule
				$x_2$&-1/2 &3/2 &1/2 &3 \\
				$x_4$& 3/2& -1/2 &-1/2 & 0 \\
				\midrule
				$\bar{\gamma}^1$&16&34&6&-19/3\\
				$\bar{\gamma}^2$&-9&-22&-2&1/5\\
				\midrule
				$\bar{f}_1$&61&-55&-26&33\\
				$\bar{f}_2$&297/2&-425/2&-153/2&-738/2\\
				$\bar{f}_3$&86&-22&1&33\\
				\bottomrule
			\end{tabular}
			\caption{ Optimal simplex table for node $N_0$}
			\label{table:lfp0}
		\end{table}
		
		As the optimal solution $ x^{*(0)} $ is integer, we apply the efficiency test  ($\mathcal{T}^1_{x^{*(0)}}$). Since the objective function gives a nonzero value at optimality, then $x^{*(0)}\notin \mathcal{X}_{Eff}$. From table [\ref{table:lfp0}], the sets $ \mathcal{H}_0=\{3,5\}$ and $ \mathcal{H'}_0=\{1,3,5\}$, so the cuts:
		$$N_1: x_3+x_5\ge 1  \;\text{and }\;x_1+x_3+x_5\ge 1$$
		are added to the table[\ref{table:lfp0}] to obtain the node $ N_1$
		\item \textbf{node $N_1$}: the program $(LFP)_1 $ gives Table[\ref{table:lfp1}]
		
		\centering
		\begin{table}[H]
			\centering
			\begin{tabular}{ccccc}
				\toprule
				$\mathcal{B}_2$ & $x_1$ &$x_3$ &$x_7$&$RHS$ \\
				\midrule
				$x_2$&-1 &1 &1/2 &5/2 \\
				$x_4$& 2& 0 &-1/2 &1/2 \\
				$x_5$& 1&1 &-1 &  1\\
				$x_6$&1 & 0&-1 &  0\\
				\midrule
				$\bar{\gamma}^1$&8&26&6&-17/2\\
				\bottomrule
			\end{tabular}
			\caption{Optimal simplex table for node $ N_1$}
			\label{table:lfp1}
		\end{table}
		
		Since the optimal solution $ x^{*(1)}=(0,\frac{5}{2},0)$ is not an integer, we apply the branching process and two nodes $ N_2$ and $ N_3$ are created, with the following constraints:
		$$ N_3:x_2 \ge \lceil \frac{5}{2} \rceil $$
		$$ N_2:x_2\le \lfloor \frac{5}{2} \rfloor$$
		\item \textbf{node $  N_3$}: the added constraint $x_2 \ge \lceil \frac{5}{2} \rceil $ makes the augmented program $(LFP)_3 $ infeasible, then the node is fathomed.
		\item \textbf{node $  N_2$}: the constraint $ x_2\le \lfloor \frac{5}{2} \rfloor $ is added to the table [\ref{table:lfp1}] to obtain, after solving $(LFP)_2 $, the table [\ref{table:lfp2}] with solution $ x^{*(2)}=(0,2,0)$ .
		
		\begin{table}[H]
			\centering
			\begin{tabular}{ccccc}
				\toprule
				$\mathcal{B}_6$ & $x_1$ &$x_3$ &$x_8$&$RHS$ \\
				\midrule
				$x_2$&0 &0 &1 &2 \\
				$x_4$& 1& 1 &-1 & 1 \\
				$x_5$& -1 &3 &-2 & 2\\
				$x_6$&-1 &2 &-2 &1 \\
				$x_7$&-2 &2 &-2&1 \\
				\midrule
				$\bar{\gamma}^1$&18&12&12&-5\\
				$\bar{\gamma}^2$&-9&-12&-4&0\\
				\midrule
				$\bar{f}_1$&-8&3&-10&64\\
				$\bar{f}_2$&50&-1&-132&-222\\
				$\bar{f}_3$&70&-40&8&28\\
				\bottomrule
			\end{tabular}
			\caption{Optimal simplex table for node $N_2$ }
			\label{table:lfp2}
		\end{table}
		
		The solution $  x^{*(2)}$ is an integer, we test the efficiency by solving ($\mathcal{T}^1_{x^{*(2)}}$). The solution is not efficient; however, $  x^{*(2)} \notin \mathcal{X}_{Eff}$ , From Table [\ref{table:lfp2}], the sets $\mathcal{H}_2=\mathcal{H'}_2=\{1,3,8\} $ then we apply the efficient cut $$ N_4 :x_1+x_3+x_8\ge 1 $$.
		\item \textbf{node $N_4 $}
		The results of the resolution of $(LFP)_4 $ are shown in table [\ref{table:lfp3}].
		
		\begin{table}[H]
			\centering
			\begin{tabular}{ccccc}
				\toprule
				$\mathcal{B}_2$ & $x_1$ &$x_6$ &$x_9$&$RHS$ \\
				\midrule
				$x_2$&-3/4 &1/4 &1/2 &7/4 \\
				$x_3$& 1/4& 1/4 &-1/2 &3/4 \\
				$x_4$& 3/2&-1/2 &0 &  1/2\\
				$x_5$&-1/4 &-5/4 &1/2 &1/4 \\
				$x_7$& -1& -1 &0 &0 \\
				$x_8$& 3/4&-1/4 &-1/2 & 1/4\\
				\midrule
				$\bar{\gamma}^1$&9/2&1&13&-59/15\\
				\bottomrule
			\end{tabular}
			\caption{Optimal simplex table for node $ N_4$}
			\label{table:lfp3}
		\end{table}
		
		The solution $x^{*(4)}$ is not integer; the branching process is also applied with the following constraints:
		$$ N_5:x_2 \le \lceil \frac{7}{4} \rceil $$
		$$ N_6:x_2\ge \lfloor \frac{7}{4} \rfloor$$
		After solving the corresponding problems $(LFP)_5$ and $(LFP)_6$, we obtain the following results:
		
		\item \textbf{node $N_6$}: the obtained optimal solution $x^{*(6)}=(\frac{1}{3},2,\frac{2}{3})$. As $x^{*(6)}$ is not an integer, a branching process is applied. All descending nodes of the solution search tree do not contain any efficient solutions (see [\ref{Tree1}]).
		\item \textbf{node $N_5$}: the obtained optimal solution $x^{*(5)}=(0,1,0)$ is integer, so we solve the efficiency tests ($\mathcal{T}^1_{x^{*(5)}}$) and ($\mathcal{T}^2_{x^{*(5)}}$). Both objectives give zero value at optimality. Hence, $\mathcal{X}_{Eff}=\mathcal{X}_{Eff}\cup \{x^{*(5)} \}$. From table [\ref{table:lfp4}], the set $\mathcal{H}_5=\{1,9,10\}$ and $\mathcal{H'}_5=\{9\}$. The efficient cuts $$ N_7:x_1+x_9+x_{10}\ge 1  \,\text{and}\, x_9\ge1 $$
		are added to table [\ref{table:lfp4}] to get the program $(LFP)_7$.
		
		\begin{table}[H]
			\centering
			\begin{tabular}{ccccc}
				\toprule
				$\mathcal{B}_6$ & $x_1$ &$x_9$ &$x_{10}$&$RHS$ \\
				\midrule
				$x_2$&0 &0 &1 &1 \\
				$x_3$&1& -1 &1 & 0 \\
				$x_4$&0 &1 &-2 &2\\
				$x_5$&-4 &3 &-5 &4 \\
				$x_6$&-3 &2 &-4&3 \\
				$x_7$&-4 &2 &-4 &3 \\
				$x_8$&0 &0 &-1 &1 \\
				\midrule
				$\bar{\gamma}^1$&6&8&4&-11/3\\
				$\bar{\gamma}^2$&3&-8&4&-1/3\\
				\midrule
				$\bar{f}_1$&-34&-17&49&53\\
				$\bar{f}_2$&47&-19&-92&-201/2\\
				$\bar{f}_3$&108&-55&69&17\\
				\bottomrule
			\end{tabular}
			\caption{Optimal simplex table for node $N_5$}
			\label{table:lfp4}
		\end{table}
		
		\item \textbf{node $N_7$}: After solving $(LFP)_7$, we obtain the optimal integer solution $x^{*(7)}=(0,1,1)$ which is efficient for ($\mathcal{T}^1_{x^{*(7)}}$) and ($\mathcal{T}^2_{x^{*(7)}}$). Hence, $\mathcal{X}_{Eff}=\mathcal{X}_{Eff}\cup \{x^{*(7)}\}$. From table [\ref{table:lfp7}], the set $\mathcal{H}_7=\{10,11,12\}$ and $\mathcal{H'}_6=\{10,11\}$. The efficient cuts $ x_{10}+x_{11}+x_{12}\ge 1  \,\text{and}\, x_{10}+x_{11}\ge 1$
		are added to table [\ref{table:lfp4}] in order to get the next nodes.
		
		\begin{table}[H]
			\centering
			\begin{tabular}{ccccc}
				\toprule
				$\mathcal{B}_6$ & $x_{10}$ &$x_{11}$ &$x_{12}$&$RHS$ \\
				\midrule
				$x_1$&1 &-1 &1 &0 \\
				$x_2$&1& 0 &0 & 1 \\
				$x_3$&0 &1 &-2 &1\\
				$x_4$&-2 &0 &1 &1 \\
				$x_5$&-1 &-4 &7&1 \\
				$x_6$&-1 &-3 &5 &1 \\
				$x_7$&0 &-4 &6 &1 \\
				$x_8$&-1 &0 &0 &1 \\
				$x_9$&0& 0 &-1 & 1 \\
				\midrule
				$\bar{\gamma}^1$&0&8&0&-3\\
				$\bar{\gamma}^2$&-4&4&-12&-1\\
				\midrule
				$\bar{f}_1$&43&-25&19&129/2\\
				$\bar{f}_2$&-182&30&-7&-205/2\\
				$\bar{f}_3$&-97&-113&-130&53\\
				\bottomrule
			\end{tabular}
			\caption{Optimal simplex table for node $N_7$}
			\label{table:lfp7}
		\end{table}
		
		The resolution process for the next nodes is presented in the following solution search tree.
	\end{itemize}
	
\begin{figure}
  \centering
  \includegraphics[width=10cm]{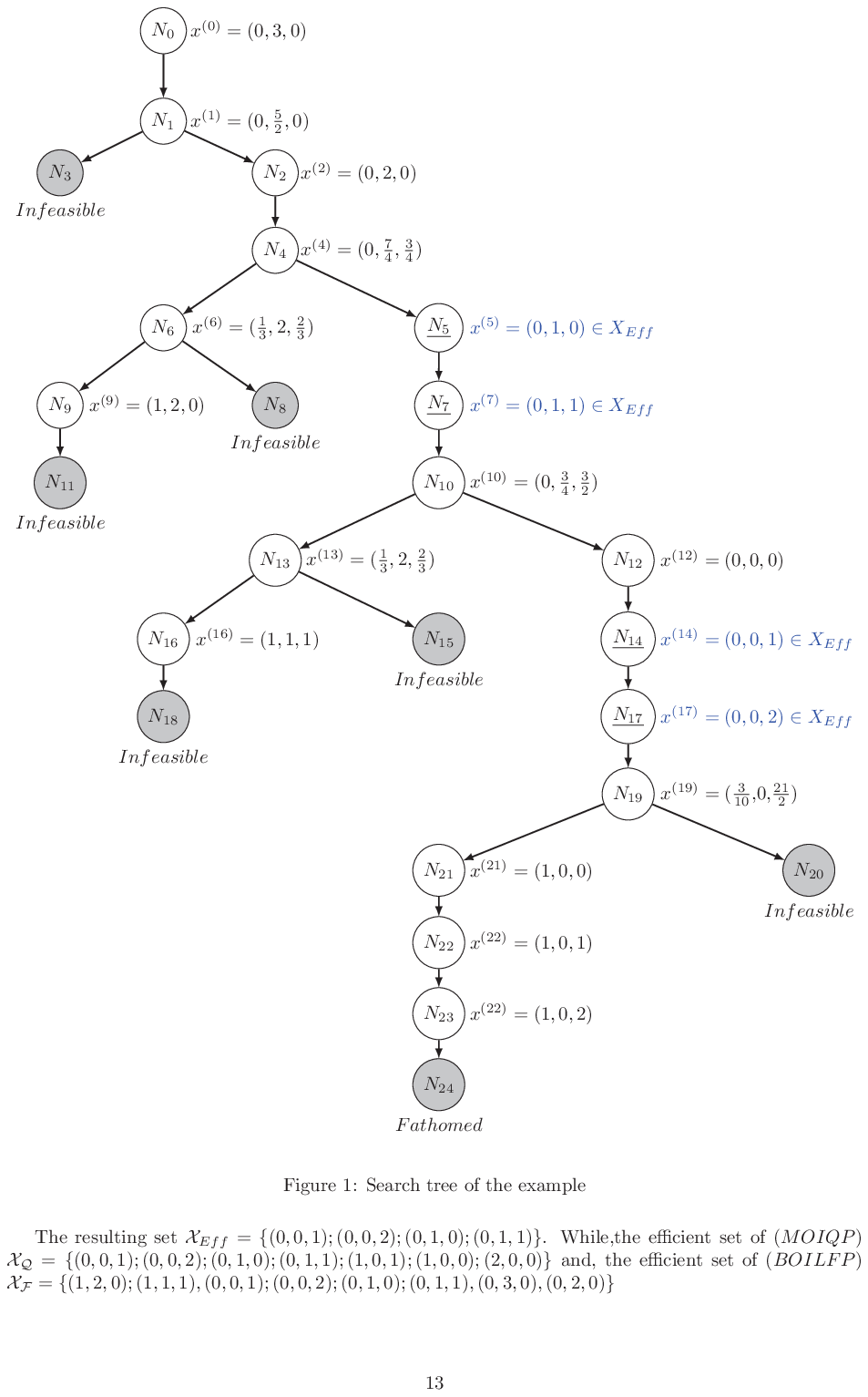}\\
  \caption{Search tree of the example}\label{Tree}
\end{figure}

	The resulting set $\mathcal{X}_{Eff} =\{(0,0,1);(0,0,2);(0,1,0);(0,1,1)\}$. While,the efficient set of $(MOIQP)$  $ \mathcal{X}_{\mathcal{Q}}=\{ (0,0,1);(0,0,2);(0,1,0);(0,1,1);(1,0,1);(1,0,0);(2,0,0)\}$ and, the efficient set of $(BOILFP)$ $\mathcal{X}_{\mathcal{F}}=\{ (1,2,0);(1,1,1),(0,0,1);(0,0,2);(0,1,0);(0,1,1),(0,3,0),(0,2,0)\} $
	\section{Conclusion}
	In this paper, we presented an exact algorithm designed for optimizing two fractional linear preference functions over the efficient set of a multi-objective integer quadratic problem. The flexibility of the algorithm allows for easy modification to optimize two linear preference functions as well. Our approach lays the foundation for addressing more challenging decision problems, particularly in the realms of game theory and other optimization fields. Future work will focus on extending the applicability of the algorithm to tackle diverse problem scenarios and contribute to advancements in decision-making methodologies.

	\bibliographystyle{apalike}

\end{document}